\documentclass[12pt]{amsart}
\usepackage{graphicx} 
\usepackage{amsfonts}
\usepackage{amsmath}
\usepackage{amssymb}
\usepackage{amsthm}
\usepackage{quiver}
\usepackage{cancel}
\usepackage{mathtools}
\usepackage{comment}
\usepackage{xcolor}
\usepackage{microtype}
\usepackage{enumitem}
\usepackage{aliascnt}
\usepackage[margin=1.25 in]{geometry}
\newtheorem{theorem}{Theorem}[section]
\newtheorem*{theorem*}{Theorem}

\newaliascnt{lemma}{theorem}
\newtheorem{lemma}[lemma]{Lemma}
\aliascntresetthe{lemma}

\newaliascnt{corollary}{theorem}
\newtheorem{corollary}[corollary]{Corollary}
\aliascntresetthe{corollary}

\newaliascnt{proposition}{theorem}
\newtheorem{proposition}[proposition]{Proposition}
\aliascntresetthe{proposition}

\theoremstyle{definition}

\newaliascnt{remark}{theorem}
\newtheorem{remark}[remark]{Remark}
\aliascntresetthe{remark}

\newaliascnt{definition}{theorem}
\newtheorem{definition}[definition]{Definition}
\aliascntresetthe{definition}

\newaliascnt{notation}{theorem}
\newtheorem{notation}[notation]{Notation}
\aliascntresetthe{notation}

\newaliascnt{example}{theorem}
\newtheorem{example}[example]{Example}
\aliascntresetthe{example}

\usepackage{aliascnt}
\usepackage{thmtools}
\usepackage{hyperref}
\usepackage{cleveref}
\newcommand{\fs}{\Pi_1(X_{\gamma(0)})}
\newcommand{\fe}{\Pi_1(X_{\gamma(1)})}
\newcommand{\fpt}{\Pi_1(\Gamma^*X)}
\newcommand{\fpts}{\Pi_1(\Gamma'^*X)}
\newcommand{\fh}{\Pi_1(H^*X)}

\newcommand{\fsl}{\Pi_1(X_{\gamma(0)}, D_{\gamma(0)})}
\newcommand{\fel}{\Pi_1(X_{\gamma(1)}, D_{\gamma(1)})}
\newcommand{\fml}{\Pi_1(X_{\gamma(t)}, D_{\gamma(t)})}
\newcommand{\fptl}{\Pi_1(\Gamma^*X, \log \Gamma^*D)}
\newcommand{\fptsl}{\Pi_1(\Gamma'^*X, \log \Gamma'^*D)}
\newcommand{\fhl}{\Pi_1(H^*X, \log H^*D)}
\newcommand{\ffl}{\Pi_1(X_\tau, D_\tau)}

\title[Categorified isomonodromic deformations via Lie groupoids]{Categorified isomonodromic deformations via Lie groupoids I: logarithmic singularities}
\author{Waleed Qaisar}

\begin{document}

\begin{abstract}
    We upgrade the classical operation of \textit{isomonodromic deformations} along a path $\gamma$ to a functor $\mathbb{P}_{\gamma}$ between categories of flat connections with logarithmic singularities along a divisor $D$, which itself depends functorially on $\gamma$, using tools from the theory of Lie groupoids. As applications, (1) we get that isomonodromy gives a map of moduli \textit{stacks} of flat connections with logarithmic singularities, (2) we encode higher homotopical information at level 2, i.e. we get an action of the fundamental 2-groupoid of the base of our family on the categories of logarithmic flat connections on the fibres, and (3) our methods produce a geometric incarnation of the isomonodromy functors as Morita equivalences which are more primary than the isomonodromy functors themselves, and from which they can be formally extracted by passing to representation categories.
\end{abstract}

\maketitle

\tableofcontents

\section{Introduction}\label{sec: intro}

\subsection{Background}

Flat connections on a complex algebraic variety or a complex manifold are subject to a remarkable operation of \textit{isomonodromic deformation} or \textit{isomonodromy}  as one varies the complex structure on the space. A classical example of isomonodromy is given by parametrized families of ordinary differential equations (ODEs), whose monodromy remains constant as we vary the parameters; the Painlev\'e equations, for example, are produced in this set-up.

More generally, given a family of complex varieties or manifolds $X \to S$, and a path $\gamma$ in $S$, one deforms a flat connection $\nabla_0$ on a fibre $X_{\gamma(0)}$ along $\gamma$ to a flat connection $\nabla_1$ on another fibre $X_{\gamma(1)}$ while keeping the underlying monodromy representation associated to the flat connection constant - $\nabla_1$ is then called the \textit{isomonodromic deformation} of $\nabla_0$ to $X_{\gamma(1)}$. This gives a map of moduli spaces of flat connections $\mathcal{M}_{dR}(X_{\gamma(0)}) \to \mathcal{M}_{dR}(X_{\gamma(1)})$ depending only on the homotopy class of $\gamma$. This operation becomes even more interesting when one is deforming flat connections with \textit{singularities} along a divisor $D$ (in fact the Painlev\'e equations are actually produced in \textit{this} context), either reduced (corresponding to logarithmic singularities) or non-reduced (corresponding to irregular singularities).

\subsection{The problem of producing functorial isomonodromy}
Seemingly irrelevant choices have to be made for this construction, even in the nonsingular case: to isomonodromically deform $\nabla_0$, one \textit{picks a lift} $\alpha$ of the path $\gamma$ from $S$ to $X$ (which connects the fibres $X_{\gamma(0)}$ and $X_{\gamma(1)}$), conjugates by this lift to produce an isomorphism of the fundamental groups of the fibres, and then pulls back the associated monodromy representation of $\nabla_0$ by this isomorphism to get a representation of $\pi_1(X_{\gamma(1)}, \alpha(1))$, which by the Riemann-Hilbert correspondence gives a flat connection $\nabla_1$ on $X_{\gamma(1)}$. As a result of making this choice of a lift of path, from this perspective isomonodromy cannot be understood as a functor between the categories of flat connections:

\begin{center}
FlatConn$(X_{\gamma(0)}) \to$ FlatConn$(X_{\gamma(1)})$
    
\end{center}

In fact, there is even more structure present here that we would like to encode - these functors should themselves depend functorially on the path $\gamma$ picked in the base.

The problem becomes worse when one considers the singular case, as the extra choices involved are more complicated than a choice of lift of the path $\gamma$. As a consequence of not being able to upgrade isomonodromy to a functor, we cannot from this classical point of view understand isomonodromy as giving maps of moduli \textit{stacks} of flat connections with singularities along $D$ - the above classical picture only gives maps of coarse moduli spaces.

In this paper, we solve the problem of upgrading isomonodromy for logarithmic flat connections (i.e. for reduced $D$) on complex curves to a functor associated to $\gamma$, which itself depends functorially on $\gamma$. The case of irregular singularities will be taken up in a sequel paper. To explain the ideas involved, we will first do the same for the nonsingular case (i.e. for $D$ being empty). In fact, in the nonsingular case our theorem applies more generally to the fibres being complex manifolds of any dimension, rather than just complex curves. We will set this up so that the arguments in the nonsingular case can be straightforwardly converted to the arguments in the logarithmic case by replacing the notion of fundamental groupoid $\Pi_1(X)$ by \textit{twisted} fundamental groupoid $\Pi_1(X,D)$. As an application, we will actually encode more information than could be seen previously: we associate to every actual path $\gamma$ a functor between the categories of flat connections on the fibre (previously the map of moduli spaces was associated only to the homotopy class $[\gamma]$), and to every (homotopy class of) homotopy $\gamma \overset{H}{\sim} {\gamma'}$ a natural isomorphism $\eta_H$ of these functors, thus getting an action of $\Pi_{\leq 2}(S)$ on flat connections on the fibres of the family. This will be packaged into the following main theorem:

\begin{theorem}\label{thm: main theorem}

    a) Let $(X,D) \to S$ be a family of complex curves with divisor, given by a holomorphic submersion $p:X \to S$ with 1-dimensional fibres which is a topological fibre bundle (for example a proper submersion), and $D$ a reduced smooth divisor in $X$ such the restriction $p|_D: D \to S$ is a covering map. From this we can produce a 2-functor $\mathbb{P}: \Pi_{\leq 2}(S) \to \text{LieGrpd}[\text{Mor}^{-1}]$ given by the diagram below, where the codomain is the 2-category of Lie groupoids with Morita equivalences as the 1-morphisms.
    
    b) When this 2-functor is post-composed with the canonical 2-functor \\ $\text{Rep}: \text{LieGrpd}[\text{Mor}^{-1}] \to \text{Cat}$ taking a Lie groupoid to its category of representations, we get isomonodromic deformations of logarithmic flat connections with singularities along $D$.
    
    \noindent In the case that $D$ is empty, this recovers nonsingular isomonodromic deformations.
    
\end{theorem}

\bigskip

\begin{table}[h] 
\centering

\resizebox{\textwidth}{!}{%
    \begin{tabular}{ccccc}

    \underline{\parbox{3cm}{\centering Fundamental \\ 2-groupoid of $S$}} & &
    \underline{\parbox{5cm}{\centering 2-category of Lie groupoids \\ with Morita equivalences}} & &
    \underline{\parbox{3.5cm}{\centering 2-category of \\ Categories}} \\

    \\ 

    \text{point } $s \in S$ &
    $\rightsquigarrow$ &
    \begin{tabular}{c}
      \small twisted fundamental groupoid of the fibre \\
      $\Pi_1(X_s, D_s)$
    \end{tabular} &
    $\rightsquigarrow$ &
    \begin{tabular}{c}
      \small representation category \\
      $\text{Rep}(\Pi_1(X_s, D_s))$
    \end{tabular} \\

    \\ 

    \text{path } $ \gamma: s_0 \to s_1 $ &
    $\rightsquigarrow$ &
    \begin{tabular}{c}
        \small Morita equivalence \\
        $ \Pi_1(X_{s_0}, D_{s_0}) \xleftarrow{\mu_0} Q_\gamma \xrightarrow{\mu_1} \Pi_1(X_{s_1}, D_{s_1}) $
    \end{tabular} &
    $\rightsquigarrow$ &
    \begin{tabular}{c}
        \small functor \\
        $ \mathbb{P}_\gamma: \text{Rep}(\Pi_1(X_{s_0}, D_{s_0})) $ \\
        $ \to \text{Rep}(\Pi_1(X_{s_1}, D_{s_1})) $
    \end{tabular} \\

    \\
    
    \text{homotopy } $ \Downarrow H $ & $\rightsquigarrow$ &
    \begin{tabular}{c}
        \small morphism of Morita equivalences \\
        $ \Downarrow \Phi_H $
    \end{tabular} & $\rightsquigarrow$ &
    \begin{tabular}{c}
        \small natural isomorphism \\
        $ \Downarrow \eta_H $
    \end{tabular} \\
    
    \\

    \text{path } $ \gamma': s_0 \to s_1 $ &
    $\rightsquigarrow$ &
    \begin{tabular}{c}
        \small Morita equivalence \\
        $ \Pi_1(X_{s_0}, D_{s_0}) \xleftarrow{\mu_0} Q_{\gamma'} \xrightarrow{\mu_1} \Pi_1(X_{s_1}, D_{s_1}) $
    \end{tabular} &
    $\rightsquigarrow$ &
    \begin{tabular}{c}
        \small functor \\
        $ \mathbb{P}_{\gamma'}: \text{Rep}(\Pi_1(X_{s_0}, D_{s_0})) $ \\
        $ \to \text{Rep}(\Pi_1(X_{s_1}, D_{s_1})) $
    \end{tabular} \\

    \end{tabular}
} 
\bigskip
\caption{The 2-functor categorifying isomonodromic deformations.}
\end{table}

The notion of Morita equivalence will be explained below, but here we just make the observation that a Morita equivalence between two Lie groupoids is given by a geometric object (a complex manifold with commuting actions of both groupoids), and so the middle column gives a geometric incarnation of the isomonodromy functors. From this we can formally extract the isomonodromy functors by taking representation categories to pass to the third column. 
\subsection{Relation to other work}

One may think of the Morita equivalences constructed here as \textit{geometric} versions of some Fourier-Mukai transforms specially adapted to go between categories of flat connections with prescribed singularities (rather than between all D-modules), which specifically come from isomonodromy.
    
The category LieGrpd[Mor$^{-1}$] (which will be explained below) of holomorphic Lie groupoids and Morita equivalences can be taken as a presentation of the category of geometric holomorphic stacks (see \cite{Behrend and Xu, Stacky Lie groups} or \cite{del Hoyo}), where a Morita equivalence of Lie groupoids gives an isomorphism of the stacks presented by the Lie groupoids.

The construction thus provides, for a given family $(X,D) \to S$, a 2-local system on $S$ valued in holomorphic stacks. The study of higher local systems is an active research area, but to our knowledge these are usually valued in some kind of `linear' or stable categories, whereas here one is producing a 2-local system \textit{valued in holomorphic stacks} by upgrading the classical geometric phenomenon of isomonodromy.

\subsection{Review of contents}
In \Cref{sec: Morita equivalences}, we review the notions of Lie groupoid, Morita equivalence, and Lie algebroid. We also collect together some propositions relating Morita equivalences and representations of groupoids, as well as a lemma we will use to construct Morita equivalences.

\Cref{sec:nonsingular} details the categorification of isomonodromic deformations of nonsingular flat connections. We start by detailing the complex manifold underlying the Morita equivalence associated to a path in \Cref{subsec: complex manifold underlying Morita equivalence}. In \Cref{subsec:groupoid action on Q_gamma} we construct the groupoid actions on this complex manifold and prove that this produces a Morita equivalence. We follow this up in \Cref{subsec:morphism of Morita equivalences} by constructing a morphism of Morita equivalences from a homotopy of paths.

In \Cref{sec: logarithmic} we detail the categorification of isomonodromic deformations of flat connections with logarithmic singularities. We preface this in \Cref{sec: twisted fundamental groupoids} by giving a review of the notion of \textit{twisted} fundamental groupoid, which play the role played by the usual fundamental groupoid in \Cref{sec:nonsingular}. In \Cref{subsec:log Q_gamma} we construct the Morita equivalence associated to a path in the logarithmic case, and in \Cref{subsec: log morphism} we construct the morphism of Morita equivalences associated to a homotopy of paths in the logarithmic case, allowing us to conclude the proof of \Cref{thm: main theorem}.

\Cref{sec: Hirzebruch surface} specializes the construction to the simplest case of categorified isomonodromy along a constant path, noting that this gives an action of the second homotopy group of the base on the category of logarithmic flat connections on a fibre. We remark that if we puncture away the singular data along the divisor, this action of the second homotopy group recovers the connecting homomorphism of the fibration underlying our family. Thus in the case of logarithmic singularities this action can be understood as generalizing the connecting homomorphism to include singular data along a divisor. We finish this section by applying the categorified isomonodromy construction to the first Hirzebruch surface to get an action of $\pi_2(\mathbb{P}^1, s)$ on the category of flat connections on $\mathbb{P}^1$ with logarithmic singularities at $0$ and $\infty$.

We end with some remarks on future work in \Cref{sec: future work}.

\subsection{Acknowledgements}
I would like to thank my advisor Marco Gualtieri, both for suggesting this project and then for having countless meetings discussing it. I would also like to thank Daniel Alvarez, Aidan Lindberg, Daniel Litt, Kumar Murty, Maarten Mol and Grisha Taroyan for many helpful conversations. Special acknowledgement is due to Francis Bischoff for many discussions about his work, and to Thomas Stanley for numerous detailed conversations about this project since I started working on it.

\subsection{Conventions and notation}

All Lie groupoids used are holomorphic Lie groupoids, i.e. internal groupoids in the category of complex manifolds whose source and target maps are holomorphic submersions. 

Whenever we say $g \in \mathcal{G}$ without qualification for a groupoid $\mathcal{G}$, we mean that $g \in \mathcal{G}^{(1)}$, the space of arrows of $\mathcal{G}$.

Our convention for concatenation of homotopy classes of paths is that $[\alpha].[\beta]$ means that we first travel along the path $\alpha$ and then along the path $\beta$.

All fibrations may be taken to be Serre fibrations, since all spaces we work with are at least CW-complexes.

For a map $p:X \to S$ and $s \in S$, $X_s$ denotes the fibre $p^{-1}(s)$.

\section{Morita equivalences}\label{sec: Morita equivalences}

A groupoid $\mathcal{G}$ has a set of objects $\mathcal{G}^{(0)}$ and a set of arrows $\mathcal{G}^{(1)}$, with a target and a source map from the arrows to the objects, together denoted as:
\[\begin{tikzcd}
	{\mathcal{G}^{(1)}} \\
	{\mathcal{G}^{(0)}}
	\arrow["t"', shift right=2, from=1-1, to=2-1]
	\arrow["s", shift left=2, from=1-1, to=2-1]
\end{tikzcd}\]

If these sets of objects and arrows are in fact (complex) manifolds, and the target and source maps are (holomorphic) submersions, $\mathcal{G}$ is called a (holomorphic) \textit{Lie} groupoid (all Lie groupoids in this paper are holomorphic). The fundamental groupoid of a complex manifold $\Pi_1(X)$ is an example, where the holomorphic charts on the space of arrows are induced by the local homeomorphism $\Pi_1(X) \xrightarrow{(t,s)} X \times X$ given by the pair of target and source maps of the arrows.

The usual morphisms of Lie groupoids are defined to be functors between the underlying groupoids (considered as categories) that are actually holomorphic as a map between the complex manifolds of arrows, and holomorphic as a map between the complex manifold of objects. One can localize this category with respect to a certain class of morphisms Mor which as functors are equivalences of categories, thus weakening the notion of isomorphism between Lie groupoids to another equivalence relation known as Morita equivalence.

A more geometrically understandable way of presenting Morita equivalences is based on the use of principal groupoid bundles, which are an equivalent notion \cite{Stacky Lie groups, del Hoyo, Bischoff log connections}. We reproduce the definition of Morita equivalence from this viewpoint below, following Section 4 of \cite{Bischoff log connections}. We first need to generalize the definitions of \textit{left action}, \textit{left $\mathcal{G}$-bundle} and \textit{left principal $\mathcal{G}$-bundle} from groups to groupoids.

Let $\mathcal{G}$ be a Lie groupoid with space of objects $\mathcal{G}^{(0)}$ and source and target maps denoted $s$ and $t$ respectively. To let $\mathcal{G}$ act on a complex manifold $X$, we need the data of a map $\mu: X \to M$. Then a left action of $\mathcal{G}$ on $X$ is a map:

\[ L: \mathcal{G} \times_{s, \mu} X \rightarrow X,\] \[ (g,x) \mapsto g.x
\]

such that $\mu(g.x) = t(g)$ and $g.(g'.x) = (gg').x$.

We can now define a left $\mathcal{G}$-bundle as a surjective submersion $\pi: Q \to Y$ together with a left action of $\mathcal{G}$ on $Q$ which preserves the fibres of $\pi$.
Lastly, such a left $\mathcal{G}$-bundle is \textit{principal} if the map

\[ \mathcal{G} \times_{\mathcal{G}^{(0)}} Q \to Q \times_Y Q\] \[ (g,q) \mapsto (g.q, q)
\]

is an isomorphism, i.e. if the action is free and transitive on the fibres of $\pi$.

\begin{definition}\label{def: Morita equivalence}
    A Morita equivalence between Lie groupoids $\mathcal{G}$ and $\mathcal{H}$ is a bi-principal $\mathcal{(G,H)}$ bi-bundle. This consists of a complex manifold $Q$, together with actions of $\mathcal{G}$ and $\mathcal{H}$ with the following properties:

\[\begin{tikzcd}
	{\mathcal{G}^{(1)}} & Q & {\mathcal{H}^{(1)}} \\
	{\mathcal{G}^{(0)}} && {\mathcal{H}^{(0)}}
	\arrow[curve={height=-18pt}, from=1-1, to=1-2]
	\arrow[shift left, from=1-1, to=2-1]
	\arrow[shift right, from=1-1, to=2-1]
	\arrow["p", from=1-2, to=2-1]
	\arrow["q"', from=1-2, to=2-3]
	\arrow[curve={height=18pt}, from=1-3, to=1-2]
	\arrow[shift right, from=1-3, to=2-3]
	\arrow[shift left, from=1-3, to=2-3]
\end{tikzcd}\]

\begin{enumerate}
    \item $q: Q \to \mathcal{H}^{(0)}$ is a left principal $\mathcal{G}$-bundle with moment map $p$.
    \item $p: Q \to \mathcal{G}^{(0)}$ is a right principal $\mathcal{H}$-bundle with moment map $q$.
    \item The actions commute, i.e. $(g.z).h = g.(z.h)$, where
    $z \in Q, g \in \mathcal{G}, h \in \mathcal{H}$ and both actions are defined.
\end{enumerate}

\end{definition}

The presentation of Morita equivalences as \emph{biprincipal groupoid-bibundles} leads to a natural notion of morphism between them.

\begin{definition}\label{def:morphism of Morita equivalences}
Let $Q$ and $Q'$ both be Morita equivalences between the groupoids $\mathcal{G}$ and $\mathcal{H}$. A morphism between them is given by a map $\eta: Q \to Q'$ of the underlying (complex) manifolds which is biequivariant for the groupoid actions, i.e. for $g \in \mathcal{G}$ and $h \in \mathcal{H}$, $\eta(g.q) = g.(\eta(q))$ and $\eta(q.h) = (\eta(q)).h$. 
    
\end{definition}

\begin{remark}\label{rem: morphism is isomorphism}

Every such morphism of Morita equivalences is in fact automatically biholomorphic by the biequivariantness condition \cite{del Hoyo}, and thus gives an \textit{isomorphism} of Morita equivalences.
    
\end{remark}

We can now define the 2-category $\text{LieGrpd}[\text{Mor}^{-1}]$ mentioned above to have as objects Lie groupoids, as 1-morphisms Morita equivalences, and as 2-morphisms morphisms of Morita equivalences.  The composition of a $(\mathcal{G}, \mathcal{H})$-Morita equivalence $Q$ and an $(\mathcal{H}, \mathcal{K})$-Morita equivalence $Q'$ is the quotient space $(Q \times_{\mathcal{H}^{(0)}} Q') / \mathcal{H}$, formed by the fiber product over the intermediate base $\mathcal{H}^{(0)}$ modulo the diagonal $\mathcal{H}$-action, leaving the outer $\mathcal{G}$ and $\mathcal{K}$ actions intact. The composition of morphisms of Morita equivalences is just given by compositions of the underlying maps of complex manifolds. For more details about compositions of Morita equivalences and morphisms of Morita equivalences, see \cite{Stacky Lie groups, del Hoyo}.

The usefulness of the notion of Morita equivalence of groupoids is that it induces an equivalence between the corresponding categories of \textit{representations} of the groupoids. A representation of a Lie groupoid $\mathcal{G}$ can be defined as a vector bundle $V \to \mathcal{G}^{(0)}$ together with a left action of $\mathcal{G}$ on $V$ which is fibrewise linear. Analogously to the case of groups, we can also view a representation as a groupoid homomorphism $\mathcal{G} \to \mathcal{GL}(V)$, where $\mathcal{GL}(V)$ is the groupoid of all linear isomorphisms between the fibres of $V$.

We record the following proposition relating Morita equivalences and representations here:

\begin{proposition}\label{prop: Morita equivalence and representations} \cite{Bischoff log connections} (Proposition 4.2) A Morita equivalence $Q$ between groupoids $\mathcal{G}$ and $\mathcal{H}$ induces an equivalence between their categories of representations
Rep$(\mathcal{G}) \cong $ Rep$(\mathcal{H})$. 

\end{proposition}

The functor giving the equivalence of categories in the above proposition can be defined on objects by a pull-push procedure that transfers vector bundles on $\mathcal{G}^{(0)}$ with $\mathcal{G}$-action to vector bundles on $\mathcal{H}^{(0)}$ with $\mathcal{H}$-action, pulling back $V$ by the moment map $p: Q \to \mathcal{G}^{(0)}$ (whose fibres have a free and fibrewise-transitive $\mathcal{H}$-action), and then quotienting out by the $\mathcal{G}$-action. For more details, see \cite{Bischoff log connections} or \cite{moerdijk mrcun book}.

We also sketch a proof of the following well-known proposition relating morphisms of Morita equivalences to natural isomorphisms of functors between representation categories (see for instance \cite{del Hoyo}):

\begin{proposition}\label{prop: morphism of Morita equivalences induces natural isomorphism}
A morphism of Morita equivalences $\eta: Q \to Q'$, with $Q$ and $Q'$ both Morita equivalences between $\mathcal{G}$ and $\mathcal{H}$, induces a natural isomorphism between the functors between representation categories induced by $Q$ and $Q'$.
    
\end{proposition}

\begin{proof}
Note that $\eta$ is actually an \textit{isomorphism}, by \Cref{rem: morphism is isomorphism}. Letting $V$ denote an object of $\text{Rep}(\mathcal{G)}$, we have that since $\eta$ respects the moment maps ($p_{Q'} \circ \eta = p_Q$), it lifts canonically to a bundle isomorphism $\tilde{\eta}_V: p_Q^*V \to p_{Q'}^*V$ covering $\eta$ (acting as the identity on the fibers of $V$). Because $\eta$ is $\mathcal{G}$-equivariant, $\tilde{\eta}$ preserves the diagonal $\mathcal{G}$-action and descends to a well-defined isomorphism $\hat{\eta}_V: (p_Q^*V)/\mathcal{G} \xrightarrow{\sim} (p_{Q'}^*V)/\mathcal{G}$. The $\mathcal{H}$-equivariance of $\eta$ ensures that $\hat{\eta}_V$ intertwines the residual $\mathcal{H}$-actions, which gives that the family $\hat{\eta}$ defines a natural isomorphism.

\end{proof}

We will also need the infinitesimal counterpart to the notion of a Lie groupoid: a Lie algebroid. 

\begin{definition}\label{def: Lie algebroid}

A Lie algebroid over a (complex) manifold $M$ is a vector bundle $E \to M$ equipped with a Lie bracket on its space of sections, together with a map of vector bundles (called the \emph{anchor} map) $\rho: E \to TM$ which preserves the bracket, i.e.  for $X_1, X_2 \in \Gamma(M)$ we have:
\[\rho([X_1, X_2]_{\Gamma(E)}) = [\rho(X_1), \rho(X_2)]_{\Gamma(TM)}\] 

such that the Leibniz rule holds, i.e. for $f \in \Gamma(\mathcal{O}_M)$ we have:
\[ [X_1, f.X_2] = f.[X_1,X_2] + \rho(X_1)(f). X_2
\]
    
\end{definition}

Every Lie groupoid can be ``differentiated'' to a Lie algebroid by looking at infinitesimal information near the identity bisection of the Lie groupoid (the submanifold of the space of arrows $\mathcal{G}^{(1)}$ given by the identity arrows, which is isomorphic to $\mathcal{G}^{(0)}$ itself). The Lie algebroid $\text{Lie}(\mathcal{G}) = A \rightrightarrows M$ associated to a Lie groupoid $\mathcal{G}$ with space of objects $\mathcal{G}^{(0)} = M$ has underlying vector bundle given by $A = T\mathcal{G}^{(1)}|_M / TM$, the normal bundle of $M$ in $\mathcal{G}^{(1)}$. Note that the target map $t$ and source map $s$ of $\mathcal{G}$ coincide on the identity bisection $M$, which means that $Ts - Tt$ descends from $T\mathcal{G}$ to a map $A \to TM$. Just as for Lie groups and Lie algebras, one can also ``differentiate'' morphisms of Lie groupoids, and then ``$\text{Lie}$'' extends to be a functor from Lie groupoids to Lie algebroids $\text{Lie}: \text{LieGrpd} \to \text{LieAlgbd}$. For more details and an introduction to Lie groupoids and Lie algebroids, see \cite{moerdijk mrcun book} or \cite{mackenzie book}.

For a given Lie algebroid $A$, if there exists a Lie groupoid $\mathcal{G}$ such that $\text{Lie}(\mathcal{G}) = A$, we say that $A$ is \textit{integrable} and that $\mathcal{G}$ integrates $A$. What we will need about Lie algebroids is that integrable Lie algebroids $A$ (all Lie algebroids we will use in this paper are integrable) have a unique source-simply connected integration $\mathcal{G}$ (i.e. with simply connected fibres for the source map $s$) such that $\text{Lie}(\mathcal{G}) = A$, and that Lie's second theorem holds:

\begin{proposition}\label{prop: Lie II} \cite{moerdijk mrcun book} (Proposition 6.3)
    Let $\phi: \text{Lie}(\mathcal{H}) \to \text{Lie}(\mathcal{G})$ be a morphism of Lie algebroids, with $\mathcal{H}$ source-simply connected. Then there exists a unique morphism of Lie groupoids $\Phi: \mathcal{H} \to \mathcal{G}$ such that $\text{Lie}(\Phi) = \phi$.
\end{proposition}

In this case, we say that $\Phi$ \textit{integrates} $\phi$. \\

To construct Morita equivalences we will need the following well-known criterion for a Morita equivalent subgroupoid, phrased in terms of Lie algebroids:

\begin{lemma}\label{lemma: criterion for Morita subgroupoid} \cite{moerdijk mrcun book}, \cite{Bischoff log connections} (Proposition 4.5)
Let $\mathcal{G}$ be a Lie groupoid with space of objects denoted $M$, and target and source maps denoted $t$ and $s$ respectively, and let $N \subset M$ be an embedded submanifold. Let $A \rightrightarrows M$ denote the Lie algebroid of $\mathcal{G}$ with anchor map $\rho$. If $N$ intersects every orbit of $\mathcal{G}$ and is transverse to $\rho$, then $\mathcal{G}|_N = t^{-1}(N) \cap s^{-1}(N)$ is a Lie subgroupoid of $\mathcal{G}$ which is Morita equivalent to $\mathcal{G}$ via the following Morita equivalence:

\[\begin{tikzcd}
	{\mathcal{G}|_N} & {t^{-1}(N)} & {\mathcal{G}^{(1)}} \\
	N && {\mathcal{G}^{(0)}}
	\arrow[curve={height=-18pt}, from=1-1, to=1-2]
	\arrow[shift left, from=1-1, to=2-1]
	\arrow[shift right, from=1-1, to=2-1]
	\arrow["t", from=1-2, to=2-1]
	\arrow["s"', from=1-2, to=2-3]
	\arrow[curve={height=18pt}, from=1-3, to=1-2]
	\arrow[shift right, from=1-3, to=2-3]
	\arrow[shift left, from=1-3, to=2-3]
\end{tikzcd}\]

The left action here is defined as $g. \alpha \overset{\text{def}}{=} \iota(g).\alpha$, where $\iota: \mathcal{G}|_N \hookrightarrow \mathcal{G}$ is the inclusion, and the right-hand side of the equality is groupoid multiplication in $\mathcal{G}$.
\end{lemma}

\begin{remark}\label{rem:switching s and t for Morita equivalence}

We could instead have taken $s^{-1}(N)$ to get a Morita equivalence with $\mathcal{G}$ acting on the left and $\mathcal{G}|_N$ acting on the right.
    
\end{remark}

\section{Categorification of nonsingular isomonodromy}\label{sec:nonsingular}

In this section, we construct the Morita equivalence corresponding to a chosen path $\gamma$ in the nonsingular case, i.e. for a family of complex manifolds $X \to S$, where the divisor $D \subset X$ is empty. The idea for building the Morita equivalence \textit{without} choosing a lift of the path $\gamma$ is to construct the space of all possible choices of lifts and realize that this itself carries a complex structure, as well as carrying actions of both the fundamental groupoids of the fibres we are comparing. We follow this up in \Cref{subsec:morphism of Morita equivalences} by constructing the morphism of Morita equivalences corresponding to a chosen homotopy of paths $H$.

\subsection{Construction of the complex manifold underlying the Morita equivalence}\label{subsec: complex manifold underlying Morita equivalence}
Throughout this section, we fix $p: X \to S$ to be a holomorphic submersion which is also a fibration, and $\gamma: I \to S$ to be a continuous path.

\begin{remark}
    As is often the case in geometric problems, we will frequently be interested in the case where $p:X \to S$ is actually a topological \textit{fibre bundle} rather than an abstract fibration, but all of the proofs are carried out in the generality of fibrations. We note that the hypotheses include the important special case of $p:X \to S$ being a proper submersion, but we also want to allow the fibres to be non-compact.
\end{remark}

We would like to `thicken' $\gamma$ to a small complex-analytic neighbourhood of itself, so that we can do all of our constructions inside the category of complex manifolds - this will help to immediately realize the complex structure on the `space of choices of lifts' that we mention above. All of this motivates the following construction:

\textbf{Construction.} The graph $Gr(\gamma) = \{(\gamma(t), t) | t \in I \} \subset S \times I$ is contractible since it is homeomorphic to the interval $I$ ($I$ maps bijectively onto $Gr(\gamma)$ by $t \mapsto (\gamma(t), t)$, $I$ is compact, and all spaces involved are Hausdorff). $Gr(\gamma) \subset S \times I \subset S \times \mathbb{C}$, and since $Gr(\gamma)$ is contractible, it is contained in a small enough open subset $\mathbb{D} \subset S \times \mathbb{C}$ such that $\mathbb{D}$ is still contractible. $\mathbb{D}$ is thus a contractible complex manifold, and the projection map $S \times \mathbb{C} \to S$ restricts to give a map $\mathbb{D} \to S$ that we will denote by $\Gamma$ (since it has been produced from $\gamma$).

We can now form the pullback family along $\Gamma$:

\[\begin{tikzcd}
	{\Gamma^*X \overset{\mathrm{def}}{=} X \times_S \mathbb{D}} & X \\
	\mathbb{D} & S
	\arrow["\pi_X", from=1-1, to=1-2]
	\arrow["\pi"', from=1-1, to=2-1]
	\arrow["p", from=1-2, to=2-2]
	\arrow["\Gamma", from=2-1, to=2-2]
\end{tikzcd}\]

That $p$ is a holomorphic submersion ensures that $\Gamma^*X$ is a complex manifold and the projection $\pi: \Gamma^*X \to \mathbb{D}$ is a holomorphic submersion. Thus the induced map:

\[\begin{tikzcd}
	{\Pi_1(\Gamma^*X)} \\
	{\Pi_1(\mathbb{D)}}
	\arrow["{\pi_*}", from=1-1, to=2-1]
\end{tikzcd}\]

is also a holomorphic submersion (since $\pi_*$ locally around an element $[\gamma]$ looks like $\pi \times \pi$ by using the coordinate charts on the space of arrows of the fundamental groupoid).

Recalling that $(\gamma(t), t) \in Gr(\gamma) \subset \mathbb{D}$, we can make the following definition:

\begin{definition}
    Denote by $e$ the path $t \mapsto (\gamma(t), t)$ in $\mathbb{D}$. We define $Q_{\gamma} \overset{\text{def}}{=} \pi_*^{-1}([e])$.

\[\begin{tikzcd}
	{Q_{\gamma}} & {\Pi_1(\Gamma^*X)} \\
	{[e]} & {\Pi_1(\mathbb{D)}}
	\arrow[from=1-1, to=1-2]
	\arrow[from=1-1, to=2-1]
	\arrow["{\pi_*}", from=1-2, to=2-2]
	\arrow[from=2-1, to=2-2]
\end{tikzcd}\]

\end{definition}

Note that the definition of $Q_\Gamma$ does not depend on the particular choice of open contractible $\mathbb{D}$ containing $Gr(\gamma)$, since $\Gamma$ will agree on the intersection of any two such, and $e$ lies in the intersection since it lies in $Gr(\gamma)$.

The way we have constructed $Q_\gamma$, it is the preimage of a holomorphic submersive map and so is naturally a complex manifold.

\subsection{The groupoid actions on $Q_{\gamma}$}\label{subsec:groupoid action on Q_gamma}

Since $Q_{\gamma}$ is constructed as a submanifold of $\Pi_1(\Gamma^*X)$, we can construct a natural map $\mu_0$ to $X_{\gamma(0)}$ and $\mu_1$ to $X_{\gamma(1)}$ by taking the start and end points of an element.

\begin{lemma}\label{lemma: endpoints and Q_gamma}
    (1) For $[\alpha] \in Q_\gamma$, $\mu_0 ([\alpha]) \overset{\text{def}}{=} \alpha(0)$ is an element of $X_{\gamma(0)}$, and $\mu_1([\alpha]) \overset{\text{def}}{=} \alpha(1)$ is an element of $X_{\gamma(1)}$.
    
    (2) Conversely for $[\alpha] \in \fpt$ with $\alpha(0) \in X_{\gamma(0)}$ and $\alpha(1) \in X_{\gamma(1)}$, we must have that $[\alpha] \in Q_{\gamma}$.
\end{lemma}

\begin{proof}
    $\pi_*( [\alpha] ) = [e]$ by definition since $[\alpha] \in Q_\gamma$. Thus $\pi \circ \alpha$ is homotopic with fixed endpoints to $e$ in $\mathbb{D}$ and thus has the same endpoints, meaning that $(\pi (\alpha (0)) = (\gamma(0), 0)$ and $(\pi (\alpha) (1)) = (\gamma(1), 1)$. But by definition of the fibre product $\Gamma^*X$, this is exactly what it means for $\alpha(0)$ to lie in the fibre above $\gamma(0)$, i.e. $X_{\gamma(0)}$, and for $\alpha(1)$ to lie in the fibre above $\gamma(1)$, i.e. $X_{\gamma(1)}$.

    Conversely, suppose that for $[\alpha] \in \fpt$, $\alpha(0) \in X_{\gamma(0)}$ and $\alpha(1) \in X_{\gamma(1)}$. This means that $\pi(\alpha(0)) = (\gamma(0), 0) \in \mathbb{D}$ and $\pi(\alpha(1)) = (\gamma(1), 1)$. Thus $\pi_*([\alpha])$ is the homotopy class of a path with endpoints $(\gamma(0), 0)$ and $(\gamma(1), 1)$, and since $\mathbb{D}$ is contractible, there is a unique such homotopy class given by $[e]$ as defined above. Thus $[\alpha] \in \pi_*^{-1}([e]) = Q_\gamma$ as wanted.
\end{proof}

\begin{lemma}
    $\mu_0: Q_\gamma \to X_{\gamma(0)}$ and $\mu_1: Q_\gamma \to X_{\gamma(1)}$ are surjective submersions.
\end{lemma}

\begin{proof}
    Surjectivity follows from (2) of the above lemma. To see that these are submersions, note that $\mu_0$ and $\mu_1$ are the restrictions of the source map $s$ and the target map $t$ of the groupoid $\fpt$ to the submanifold $Q_\gamma$, and $Q_\gamma$ is transverse to the fibres of $s$ and $t$. 
\end{proof}

Now that we have moment maps $\mu_0$ and $\mu_1$ to the base of the groupoids $\fs$ and $\fe$ respectively, we can define the action of these groupoids on $Q_\gamma$. To do this, we first consider the inclusion of the fibres

 \begin{center}
     $X_{\gamma(0)} \xrightarrow{\iota_0} \Gamma^*X$ and $X_{\gamma(1)} \xrightarrow{\iota_1} \Gamma^*X$
 \end{center}
 
 and then take the induced maps on fundamental groupoids 
 \begin{center}
     $\fs \xrightarrow{\iota_0} \fpt$ and $\fe \xrightarrow{\iota_1} \fpt$
 \end{center}

 (where we are denoting the induced maps by the same notation $\iota_0$ and $\iota_1$ to avoid an overload of notation). Recalling that $Q_\gamma$ is a submanifold of $\fpt$, we can then define the actions in the following way:
 
\begin{definition}
    a) The \textnormal{left action} on $Q_\gamma$ by $\fs$ is given by
    $[g_0]*[\alpha] \overset{\text{def}}{=} \iota_0[g_0].[\alpha]$ for $[g_0] \in \fs$, $[\alpha] \in Q_\gamma$, i.e. we first take the image of $[g_0]$ in the fundamental groupoid of the \textit{total space} $\Gamma^*X$, and then use the groupoid multiplication of $\fpt$ to \textnormal{precompose} $[\alpha]$ by this image.
    
    b) The \textnormal{right action} on $Q_\gamma$ by $\fe$ is given by $[\alpha]*[g_1] \overset{\text{def}}{=} [\alpha]. \iota_1[g_1]$ for $[g_1] \in \fe$, $[\alpha] \in Q_\gamma$, i.e. we first take the image of $[g_1]$ in $\fpt$ and then use its groupoid multiplication to \textnormal{postcompose} $[\alpha]$ by this image.
\end{definition}

We note that these two actions obviously commute since concatenation of paths on the left commutes with concatenation of paths on the right, i.e. $([g_0]*[\alpha])*[g_1] = [g_0]*([\alpha]*[g_1])$ for every $[\alpha] \in \fpt$, $[g_0] \in \fs$ and $[g_1] \in \fe$. Thus they satisfy condition (3) in Definition 1 (of a Morita equivalence) above.

\begin{theorem}\label{Q_gamma is a Morita equivalence}
    Let $p:X \to S$ a holomorphic submersion which is also a fibration, and let $\gamma: I \to S$ be a path in $S$. Then $Q_\gamma$ defined in Definition 2 is a Morita equivalence between $\Pi_1(X_{\gamma(0)})$ and $\Pi_1(X_{\gamma(1)})$.
\end{theorem}

\begin{proof}
    We have already checked that $Q_\gamma$ is a complex manifold and constructed the two commuting actions of the groupoids mentioned. Thus what remains to be proved is that both actions are principal, i.e. that they are free and fibrewise-transitive (on the fibres of $\mu_1$ for the $\fs$ action, and on the fibres of $\mu_0$ for the $\fe$ action).

    We begin by noting that $\pi: \Gamma^*X \to \mathbb{D}$ is a fibration since it is a pullback of the fibration $p: X \to S$. Then since $\mathbb{D}$ is contractible, choosing any basepoints $x \in X_{\gamma(t)}$ and $s \in \mathbb{D}$, the long exact sequence of homotopy groups gives that the inclusion of the fibre $X_{\gamma(t)} \xrightarrow{\iota_t} X$ induces an isomorphism of fundamental groups:

    \begin{center}
        $\dots \cancelto{0}{\pi_{2}(\mathbb{D}, s)} \rightarrow \pi_1(X_{\gamma(t)}, x) \xrightarrow{\cong} \pi_1(\Gamma^*X, x) \rightarrow \cancelto{0}{\pi_1(\mathbb{D}, s)} \dots$
    \end{center}

    \textbf{Freeness.} For the left action to be free, we need that for every $[g_0], [h_0] \in \fs$ and $[\alpha] \in Q_{\gamma}$, $[g_0]*[\alpha] = [h_0]*[\alpha]$ implies that $[g_0] = [h_0]$ in $\fs$.

    By cancellation in the groupoid $\fpt$, $\iota_0([g_0]).[\alpha] = \iota([h_0]).[\alpha]$ implies $\iota_0([g_0]) = \iota_0([h_0])$.
    
    Since $\iota_0$ is a groupoid homomorphism, this implies that $\iota_0([g_0].[h_0]^{-1}) = [e_{g_0(0)}]$, where $e_{g_0(0)}$ is the constant path at $g_0(0)$. This last equality is now an equality inside the fundamental \textit{group} $\pi_1(\Gamma^*X, g_0(0))$.
    
    Letting $x = g_0(0)$, the above long exact sequence argument showed that $\iota_0$ maps $\pi_1(X_{\gamma(0)}, g_0(0))$ \textit{isomorphically} (in particular, injectively) onto $\pi_1(\Gamma^*X, g_0(0))$, which means that we can deduce from $\iota_0([g_0].[h_0]^{-1}) = [e_{g_0(0)}]$ that $[g_0].[h_0]^{-1} = [e_{g_0(0)}]$, this time this being an equality inside the fundamental group of the \textit{fibre} $\pi_1(X_{\gamma(0)}, g_0(0))$.
    Then multiplying both sides of this equality on the right by $[h_0]$ in the fundamental \textit{groupoid} $\fs$ we get that $[g_0] = [h_0]$, which is what we wanted to show for freeness of the left action.

    Exactly the same argument with $X_{\gamma(0)}$ replaced by $X_{\gamma(1)}$ and left multiplication replaced by right multiplication shows freeness of the right action as well.

    \textbf{Fibrewise-transitivity.} We now want to show that the left action is transitive on the fibres of the map $\mu_1$, which would complete proving condition (1) in \Cref{def: Morita equivalence}.

    $\mu_1: Q \rightarrow X_{\gamma(1)}$ was just the map sending a (homotopy class of) path to its end point, so the fibre of $\mu_1$ above $x$ is homotopy classes of paths which end at $x$ and are elements of $Q_\gamma$. We want to show that if two elements $[\alpha], [\alpha'] \in Q_\gamma$ lie in the same fibre (i.e. have the same end point), then they can be related by the $\fs$-action, i.e. that there is $[g_0] \in \fs$ such that $[g_0]*[\alpha] = [\alpha']$.

    So take $[\alpha], [\alpha'] \in Q_\gamma$ with the same endpoint and consider $[\delta] \overset{\text{def}}{=} [\alpha'].[\alpha^{-1}]$, which is now an element of $\Pi_1(\Gamma^*X)$ whose start point $\alpha'(0)$ and end point $\alpha^{-1}(1) = \alpha(0)$ \textit{both} lie in the fibre $X_{\gamma(0)}$. Applying the projection $\pi$, we get that $[\pi \circ \delta]$ has both start and endpoint $\gamma(0)$, thus giving a loop lying in $\pi_1(\mathbb{D}, \gamma(0))$. Since $\mathbb{D}$ is contractible, this means that there exists a homotopy $H$ between $\pi \circ \delta$ and the constant path $e_{\gamma(0)}$ at $\gamma(0)$.

    But $\pi$ is a fibration, which means we have access to the homotopy lifting property: we can lift $H$ to a homotopy $\tilde{H}$ in the total space $\Gamma^*X$ such that $\tilde{H}(s,0) = \delta(s)$. Then $\tilde{\delta} \overset{\text{def}}{=} \tilde{H}(s,1)$ (i.e. the other face of the lifted homotopy $\tilde{H}$) is a path which has to cover the constant path $e_{\gamma(0)}$ in the base (this is part of the homotopy lifting property), so $\tilde{\delta}$ has to lie entirely within the fibre $X_{\gamma(0)}$. Thus we have produced a homotopy $\tilde{H}$ in $\Gamma^*X$ between $\delta$ and a path $\tilde{\delta}$ which lies entirely in the fibre $X_{\gamma(0)}$. Now we can compute that by definition:

    $[\tilde{\delta}]*[\alpha] = 
    \iota_0([\tilde{\delta}].[\alpha]) = $ (since once considered inside $\Gamma^*X$, $\tilde{\delta}$ is homotopic to $\delta$) 
     $[\delta].[\alpha] = $ 
     
     (by definition of $\delta$) $ [\alpha'].[\alpha^{-1}].[\alpha] = [\alpha']$.

     This proves fibrewise-transitivity for the left action.

     Fibrewise-transitivity of the right action is proved by exactly the same argument, again replacing $X_{\gamma(0)}$ by $X_{\gamma(1)}$, left multiplication by right multiplication, and the fibres of $\mu_1$ by the fibres of $\mu_0$.

     This finishes the proof of conditions (2) and (3) in Definition 1, thus proving that $Q_\gamma$ is a Morita equivalence.
  
\end{proof}

\begin{corollary}\label{cor: equivalence of flat connections}
    Let $p: X \to S$ a holomorphic submersion which is also a fibration, and let $\gamma: I \to S$ be a path in $S$. There is a canonically defined equivalence of categories $\mathbb{P}_\gamma: \text{FlatConn}(X_{\gamma(0)}) \to \text{FlatConn}(X_{\gamma(1)})$, induced by $Q_\gamma$.
\end{corollary}

\begin{proof}
    Since $Q_\gamma$ is canonically defined from $\gamma$, we can apply \Cref{prop: Morita equivalence and representations} to get a canonical equivalence of categories $\text{Rep}(\fs) \to \text{Rep}(\fe)$. Then pre and post-composing this with the Riemann-Hilbert functor gives the equivalence of categories we call $\mathbb{P}_\gamma$.
\end{proof}

\subsection{Constructing the morphism of Morita equivalences}\label{subsec:morphism of Morita equivalences}

We now turn our attention to associating to each homotopy of paths (fixing endpoints) $\gamma \overset{H}{\sim} \gamma'$ a morphism of Morita equivalences that we will call $\eta_H$, canonically constructed from $H$.

The idea of the construction is the following: $Q_\gamma$ and $Q_{\gamma'}$ are spaces of possible choices of lifts for $\gamma$ and $\gamma'$ respectively, and then from the homotopy $H$ we should get a space of possible choices of lifts of $H$ into which $Q_\gamma$ and $Q_{\gamma'}$ embed.

\textbf{Construction.} We think of the homotopy $\gamma \overset{H}{\sim} \gamma'$ as given by a map $H: B^2 \to S$, where $B^2$ is the 2-dimensional disk, sending two marked points $b_0$ and $b_1$ to $\gamma(0) = \gamma'(0)$ and $\gamma(1) = \gamma'(1)$ respectively (since $H$ is a homotopy fixing endpoints). Similarly to before, the graph $Gr(H)$ is contractible since it is homeomorphic to the square $B^2$ and we have that $Gr(H) \subset S \times B^2 \subset S \times \mathbb{C}$. Since $Gr(H)$ is contractible, it is contained in a small enough open subset that we denote $\mathbb{D}^2 \subset S \times \mathbb{C} $ such that $\mathbb{D}^2$ is still contractible, and the projection map $S \times \mathbb{C} \to S$ again restricts to give a map $\mathbb{D}^2 \to S$ that we will also denote (by abuse of notation) as $H$. In fact if $\mathbb{D}_\gamma$ and $\mathbb{D}_{\gamma'}$ denote the contractible opens we chose containing $Gr(\gamma)$ and $Gr(\gamma')$ respectively in the previous construction, we can choose $\mathbb{D}^2$ this time to extend $\mathbb{D}_\gamma$ and $\mathbb{D}_{\gamma'}$.

We now form the pullback family along this map $H$:

\[\begin{tikzcd}
	{H^*X \overset{\text{def}}{=} X \times_S \mathbb{D}^2} & X \\
	{\mathbb{D}^2} & S
	\arrow[from=1-1, to=1-2]
	\arrow["{\pi_H}"', from=1-1, to=2-1]
	\arrow["p", from=1-2, to=2-2]
	\arrow["H", from=2-1, to=2-2]
\end{tikzcd}\]

\[\mathbb{D}_\gamma \hookrightarrow \mathbb{D}^2 \hookleftarrow \mathbb{D}_{\gamma'} \]
gives the inclusions
\[\Gamma^*X \hookrightarrow H^*X \hookleftarrow \Gamma'^*X
\]
(where $\Gamma': \mathbb{D}_{\gamma'} \to S$ is the map constructed from $\gamma'$ identically to how $\Gamma: \mathbb{D}_\gamma \to S$ was constructed from $\gamma$), which gives the induced maps on fundamental groupoids
\[ \fpt \xrightarrow{\iota_\gamma} \Pi_1(H^*X) \xleftarrow{\iota_{\gamma'}} \fpts \]

\begin{lemma}\label{lemma: inclusion into homotopy total space is equivalence}
    The maps induced by inclusion $\fpt \xrightarrow{\iota_\gamma} \Pi_1(H^*X)$ and $\Pi_1(H^*X) \xleftarrow{\iota_{\gamma'}} \fpts$ are fully faithful as functors.
\end{lemma}

\begin{proof}
    To prove faithfulness, we must show that for $[\alpha], [\beta] \in \fpt$, $\iota_\gamma([\alpha]) = \iota_\gamma([\beta])$ implies $[\alpha] = [\beta]$. The given equality says that there is a homotopy $h$ between $\iota_\gamma(\alpha)$ and $\iota_\gamma(\beta)$ in $H^*X$. Then $\pi_H \circ h$ is a homotopy in $\mathbb{D}^2$ whose boundary is a loop that lies entirely in $\mathbb{D}_\gamma$. Since $\mathbb{D}_\gamma$ is contractible, the loop is homotopic to a constant path in $\mathbb{D}_\gamma$. Denoting by $g$ a homotopy between this loop and constant path, we now have two maps $\pi_H \circ h$ and $g$ that share the same boundary, giving a map from $S^2$ into $\mathbb{D}^2$. Since $\pi_2(\mathbb{D}^2, s) = 0$, $\pi_H$ is homotopic to $g$ relative to their shared boundary. Letting $K$ denote the homotopy of homotopies witnessing this, we can use the homotopy lifting property to lift $K$ to $\tilde{K}$, a homotopy in $H^*X$ such that $\tilde{K}(-, 0) = h$. Letting $\tilde{h} = \tilde{K}(-,1)$, $\tilde{h}$ is now a homotopy between $\alpha$ and $\beta$ such that $\pi_H \circ \tilde{h} = g$ which takes values in $\mathbb{D}_\gamma$, i.e. the image of $\tilde{h}$ lies in $\pi_H^{-1}(\mathbb{D}_\gamma) = \Gamma^*X$. Thus $\tilde{h}$ is a homotopy between $\alpha$ and $\beta$ in $\Gamma^*X$, giving that $[\alpha] = [\beta]$ and finishing the proof of faithfulness of $\iota_\gamma$.

    We now prove fullness. Take $[\delta] \in \text{Hom}(x,y) \subset \fh|_{\Gamma^*X}$. We can apply the map induced by the projection map $\pi_H$ to get an element ${\pi_H}_*([\delta]) \in \Pi_1(\mathbb{D}^2)$ whose source $s({\pi_H}_*([\delta]))$ and target $t({\pi_H}_*([\delta]))$ both lie in $\Gamma^*X$ (since $x$ and $y$ lie in $\Gamma^*X$ and ${\pi_H}_*$ is a groupoid homomorphism). Since $\mathbb{D}_\gamma$ is path-connected, there is a path $\beta$ in $\mathbb{D}_\gamma$ connecting the points $s({\pi_H}_*([\delta]))$ and $t({\pi_H}_*([\delta]))$. Since the larger base $\mathbb{D}^2$ is itself contractible and $\pi_H \circ \delta$ and $\beta$ have the same endpoints in $\mathbb{D}^2$, there exists a homotopy $h$ between them in $\mathbb{D}^2$. Now since $\pi_H$ is a fibration, by the homotopy lifting property we can choose a lift $\tilde{h}$ of $h$ such that $\tilde{h}(0,t) = \delta(t)$ and $\pi_H \circ\tilde{h} = h$. Defining $\tilde{\delta}(t) = \tilde{h}(1,t)$, we then have that $\pi_H(\tilde{\delta}(t)) = \pi_H(\tilde{h}(t,1)) = h(t,1) = \beta(t)$. Since $\beta$ is entirely contained in $\mathbb{D}_\gamma$, we have that $\tilde{\delta}$ is entirely contained in $\pi_H^{-1}(\mathbb{D}_\gamma) = \Gamma^*X$, with $\tilde{\delta}$ homotopic to $\delta$ via $\tilde{h}$ in $H^*X$. Thus $[\tilde{\delta}] \in \fpt$ with $\iota_\gamma([\tilde{\delta}]) = [\delta]$, proving fullness. Identical arguments prove faithfulness and fullness of $\iota_{\gamma'}$ as well.
\end{proof}

We can now construct the map $\eta_H: Q_\gamma \to Q_{\gamma'}$ corresponding to the homotopy $H$.

\begin{notation}
    Below we denote by $\alpha$ a groupoid arrow of $\fpt$ rather than a path representative for some homotopy class of path like was done above. This is because the arguments we will give here do not require picking explicit path representatives at all, and it suffices to work with the abstraction of groupoid arrows. This avoids clutter and will also come in handy when we deal with the case of logarithmic singularities and replace fundamental groupoids with twisted fundamental groupoids.
\end{notation}

\begin{theorem}\label{eta_H is morphism}

Let $\fpt \xrightarrow{\iota_\gamma} \Pi_1(H^*X) \xleftarrow{\iota_{\gamma'}} \fpts$ be the maps induced by inclusion. Let $\eta_H: Q_\gamma \to Q_{\gamma'}$ be defined by $\eta_H(\alpha) = \iota_{\gamma'}^{-1} \circ \iota_\gamma(\alpha)$. Then $\eta_H$ is a morphism of Morita equivalences.
\end{theorem}

\begin{proof} We first need to check that $\eta_H$ is well-defined, and takes $Q_\gamma$ to $Q_{\gamma'}$.
    Picking $\alpha \in Q_\gamma \subset \fpt$, we can look at its image under $\iota_\gamma$ to get an element $\iota_\gamma(\alpha) \in \fh$. Note that since $\alpha \in Q_\gamma$, the source $s(\alpha)$ lies in $X_{\gamma(0)} = X_{\gamma'(0)}$ and the target $t(\alpha)$ lies in $X_{\gamma(1)} = X_{\gamma'(1)}$ by \Cref{lemma: endpoints and Q_gamma}. Since $\iota_\gamma$ is a groupoid homomorphism, this gives that $s(\iota_\gamma(\alpha)) \in X_{\gamma(0)} = X_{\gamma'(0)}$ and $t(\iota_\gamma(\alpha)) = X_{\gamma(1)} = X_{\gamma'(1)}$. Now since $X_{\gamma'(0)}, X_{\gamma'(1)} \subset \Gamma'^*X$ and $\iota_{\gamma'}$ is fully faithful by \Cref{lemma: inclusion into homotopy total space is equivalence}, $\iota_\gamma(\alpha)$ has a unique preimage in $\fpts$ under $\iota_{\gamma'}$, which is what we are defining to be $\eta(\alpha)$.
    
    Applying the projection $\pi_*: \fpts \to \mathbb{D}$ to $\eta_H(\alpha)$, we have that $s(\pi_*(\eta_H(\alpha))) = \gamma(0) = \gamma'(0)$ and $t(\pi_*(\eta_H(\alpha))) = \gamma(1) = \gamma'(1)$. Then by \Cref{lemma: endpoints and Q_gamma} again for $\fpts$, we have that $\alpha \in Q_{\gamma'}$.

 Note that this map is holomorphic since it is given by a composition of holomorphic maps. It is biequivariant since the actions are defined by inclusion and then groupoid multiplication in the groupoids $\fpt$ and $\fpts$, which are themselves included into $\fh$ by the above lemma, i.e. 
for $g \in \fs$:
\[\eta_H(g*\alpha) = \iota_{\gamma'}^{-1} \circ \iota_{\gamma}(\iota_0(g).\alpha) = \iota_0(g).(\iota_{\gamma'}^{-1} \circ \iota_\gamma(\alpha)) = g * \eta_H(\alpha)
\]
and for $h \in \fe$:
\[ \eta_H(\alpha * h) = \iota_{\gamma'}^{-1} \circ \iota_{\gamma}(\alpha . \iota_1(h)) = (\iota_{\gamma'}^{-1} \circ \iota_{\gamma}(\alpha)) . \iota_1(h) = \eta_H(\alpha) * h
\]

Thus $\eta_H$ is a biequivariant holomorphic map $Q_\gamma \to Q_{\gamma'}$, which means it is a morphism of Morita equivalences.
\end{proof}

\section{Categorification of logarithmic isomonodromy}\label{sec: logarithmic}

Using the ideas developed in the previous section, we now construct the Morita equivalence corresponding to a chosen path $\gamma$ in the logarithmic case, i.e. for a family of pairs $(X,D) \to S$ where $D$ is reduced divisor. Similarly to before, in \Cref{subsec: log morphism} we follow this up by constructing the morphism of Morita equivalences corresponding to a chosen homotopy $H$.

\subsection{Review of twisted fundamental groupoids}\label{sec: twisted fundamental groupoids}

We will replace the fundamental groupoids $\Pi_1(X_s)$ appearing in the previous section with twisted fundamental groupoids.

Given a complex manifold $X$ with a reduced smooth divisor $D$, one can consider the category of flat connections on $X$ with singularities along $D$. This category  is equivalent to the category of representations of the Lie algebroid $T_X(-\log D)$, whose sections are vector fields on $X$ which are tangent to $D$ along $D$. These Lie algebroids have source-simply connected integrations that we will denote by $\Pi_1(X, \log D)$.

In the case that $X$ is a curve with divisor $D$, the Lie algebroid $T_X(- \log D)$ is equivalent to $T_X(-D)$, the Lie algebroid whose sections are vector fields on $X$ which vanish along $D$, because in this case $D$ is 0-dimensional. The integrations of these algebroids are thus denoted by $\Pi_1(X,D)$ in \cite{Stokes groupoids}. We will use that the category of representations of $\Pi_1(X, D)$ is equivalent to the category of representations of its Lie algebroid by Lie's second theorem for Lie groupoids \Cref{prop: Lie II}. One can view this as a generalized Riemann-Hilbert correspondence for flat connections with singularities along a divisor. From this point of view, the generalized Riemann-Hilbert correspondence is the purely formal identification of meromorphic flat connections with representations of a Lie algebroid, composed with the transcendental procedure of integrating Lie algebroids to Lie groupoids.
\[ \text{FlatConn}_X(D) \cong \text{Rep}(T_X(-D))\cong\text{Rep}(\Pi_1(X,D))\]

\begin{remark}
    The integrations $\Pi_1(X, \log D)$ can have a space of arrows which is non-Hausdorff in general. In \cite{Stokes groupoids}, the authors study the case of complex curves and show that the only case in which this groupoid fails to be Hausdorff is for $(X, D) = (\mathbb{P}^1, p)$. Bischoff studies $\Pi_1(X, \log D)$ for higher dimensional $X$ without imposing the Hausdorffness restraint in \cite{Bischoff log connections}. In this paper, we will only need to study higher dimensional complex manifolds which are families of curves over a contractible base. Since the obstruction to being Hausdorff (see \cite{Stokes groupoids}) is a purely homotopical condition (in fact just a condition on fundamental groups), the Hausdorffness of our groupoids reduces to the case of curves; thus by \cite{Stokes groupoids}, we only need to exclude the case $(\mathbb{P}^1,p)$ if we wish for all our constructions to remain within the category of (Hausdorff) complex manifolds. However, even this is not necessary, as our constructions go through even allowing this one non-Hausdorff example, with the understanding that the notion of Lie groupoid may now refer to examples where the space of arrows is possibly non-Hausdorff, as is sometimes taken to be the case in the literature.
\end{remark}

\subsection{Construction of the Morita equivalence in the logarithmic case}\label{subsec:log Q_gamma}

Fix $p: X \to S$ to be a holomorphic submersion which is also a fibration, with the \textit{fibres being complex curves}, and let $D$ be a reduced smooth divisor such that the restriction $p|_D: D \to S$ is a cover of $S$. (Note that if we had let $D$ more generally be a normal crossings divisor, the condition of being a cover of $S$ will imply that there are actually no singularities.)

For a path $\gamma$ in the base $S$, we constructed in the previous section a contractible open subset $\mathbb{D} \subset S \times \mathbb{C}$, equipped with the restriction of the projection map $\Gamma: \mathbb{D} \to S$. We can now pullback the family of pairs $(X,D) \to S$ along this map $\Gamma$:

\[\begin{tikzcd}
	{(\Gamma^*X, \Gamma^*D) \overset{\text{def}}{=} (X \times_S \mathbb{D}, D \times_S \mathbb{D})} & {(X,D)} \\
	{\mathbb{D}} & S
	\arrow["{\pi_X}", from=1-1, to=1-2]
	\arrow["\pi"', from=1-1, to=2-1]
	\arrow["p", from=1-2, to=2-2]
	\arrow["\Gamma"', from=2-1, to=2-2]
\end{tikzcd}\]

The map $\pi: \Gamma^*X \to \mathbb{D}$ is again a holomorphic submersive (topological) fibration over the contractible base $\mathbb{D}$ (since it is a pullback of a holomorphic submersive fibration). The restriction $\pi|_{\Gamma^*D}: \Gamma^*D \to \mathbb{D}$ is a cover of the contractible base $\mathbb{D}$, and it is thus a trivial cover $\Gamma^*D \cong F \times \mathbb{D}$ (for $F$ a discrete space). Note that this also means that $\Gamma^*D$ is a reduced divisor in $\Gamma^*X$, since it is the pullback of a reduced divisor under $\Gamma$, which factors as an injection following by a projection.

We can then form the twisted fundamental groupoid $\fptl$. We would like to have a version of the inclusion maps $\iota_0$ and $\iota_1$ from \Cref{sec:nonsingular}. To construct these, take the inclusion of pairs
\[ (X_\tau, D_\tau) \hookrightarrow (\Gamma^*X, \Gamma^*D)
\]
and look at the induced map on Lie algebroids
\[ T_{X_\tau}(-D_\tau) \rightarrow T_{\Gamma^*X}(-\Gamma^*D)
\]
given by taking the differential of the inclusion. Note that one could also obtain the same map by looking at the short exact sequence of bundles
\[0 \to T_{\Gamma^*X/\mathbb{D}}(-\log \Gamma^*D) \to T_{\Gamma^*X}(- \log \Gamma^*D) \to \pi^*T\mathbb{D} \to 0
\]
and restricting to the fibre $X_\tau$: then $T_{\Gamma^*X/\mathbb{D}}(-\log \Gamma^*D)|_{X_\tau}$ is naturally isomorphic to $T_{X_\tau}(-D_\tau)$ (recall that $X_\tau$ is a curve), and the map above is given by the inclusion of the relative logarithmic tangent bundle into the total logarithmic tangent bundle.

We can then apply \Cref{prop: Lie II} to integrate this map of Lie algebroids to a map between the source-simply connected integrations $\fml \to \fptl$ that we will denote by $\iota_\tau$. 

Since the twisted fundamental groupoid restricted to the complement of the divisor is isomorphic to the fundamental groupoid of the complement by \cite{Stokes groupoids} (Lemma 3.12), i.e. $\ffl|_{X_\tau \backslash D_\tau}
\cong \Pi_1(X_\tau \backslash D_\tau)$, and $\iota_\tau$ was produced by integrating the differential of the inclusion map of fibres, we have that $\iota_\tau$ restricted to the complement of the divisor is just the map of fundamental groupoids induced by inclusion of the fibre:
\[\iota_\tau|_{X_\tau \backslash D_\tau} : \Pi_1(X_\tau \backslash D_\tau) \to \Pi_1(\Gamma^*X \backslash \Gamma^*D). \]

We can now construct the Morita equivalence associated to a path $\gamma$.

\begin{notation}
    We will again denote the Morita equivalence associated to a path $\gamma$ by $Q_\gamma$, just as in the nonsingular case. This should hopefully cause no confusion as the nonsingular Morita equivalences will not appear from now on.
\end{notation}

This time, we will use the criterion given in \Cref{lemma: criterion for Morita subgroupoid} to produce the Morita equivalence, rather than building the `space of choices of lifts of $\gamma$' by hand like in the nonsingular case, but the result at the end will admit a similar interpretation.

\begin{theorem}\label{thm: log Q_gamma is Morita equivalence}
    Let $t$ and $s$ denote the source and target maps of $\fptl$ respectively. Then $t^{-1}(X_\tau)$ gives a Morita equivalence between $\ffl$ and $\fptl$.
\end{theorem}

\begin{proof}
    $X_\tau$ intersects both orbits $\Gamma^*X \backslash \Gamma^*D$ (since it contains points in $X_\tau \backslash D_\tau$) and $\Gamma^*D$ (since it contains points of $D_\tau$) of the groupoid $\fptl$, and it is transverse to the anchor map $\rho$. Thus by \Cref{lemma: criterion for Morita subgroupoid}, $t^{-1}(X_\tau)$ gives a Morita equivalence between $\fptl|_{X_\tau}$ and $\fptl$. We are thus left to show that $\ffl \cong \fptl|_{X_\tau}$.

    First note that the Lie algebroid of $\fptl|_{X_\tau}$ is also given by $T_{X_\tau}(-D_\tau)$: we have an injective map from $T_{X_\tau}(-D_\tau)$ to the Lie algebroid of $\fptl|_{X_\tau}$, and this map is surjective because every vector field on $X_\tau$ that vanishes along $D_\tau$ can be locally extended  to a neighbourhood of $X_\tau$ in $\Gamma^*X$ by considering a local expression for the vector field as given by $z \partial_z$ (with $z$ a coordinate on $X_\tau$ vanishing at a chosen point in $D_\tau$). Now, take the identity map from the Lie algebroid $T_{X_\tau}(-D_\tau)$ to itself, and use \Cref{prop: Lie II} to integrate it to a map $\ffl \to \fptl|_{X_\tau}$ - by definition this is the map $\iota_\tau$ defined above. Since this map integrates the identity map on Lie algebroids, it is a covering map (since it factors through the source-simply connected cover of the codomain), and is thus surjective. 
    
    We now check injectivity of this map. Note that over the complement of the divisor $\iota_\tau|_{X_\tau \backslash D_\tau}$ agrees with the map on actual fundamental groupoids induced by inclusion, and so it is an injective map here by the results of the previous section (or since it is the map induced on fundamental groupoids by a homotopy equivalence, where the homotopy equivalence is given by the inclusion of a fibre into the total space for a fibration with a contractible base). We now have a covering map which is injective on the dense open $X_\tau \backslash D_\tau$. Now if the map sends two distinct points to the same image, it would map their disjoint local neighborhoods to the same target. Since the dense open set intersects both disjoint local neighbourhoods, it would contain distinct points from each neighbourhood mapping to the same value, contradicting the hypothesis that the map is injective on the dense open. Thus the map must be injective everywhere.
    We conclude that we have isomorphism $\ffl \xrightarrow{\cong} \fptl|_{X_\tau}$, which we can compose with the obtained Morita equivalence to get a Morita equivalence between $\ffl$ and $\fptl$.
\end{proof}

Note that by using \Cref{rem:switching s and t for Morita equivalence}, $s^{-1}(X_\tau)$ would give a Morita equivalence between $\fptl$ and $\ffl$ as well.

\begin{definition}
    For $\gamma$ a path in $S$, we can compose the Morita equivalences given by $s^{-1}(X_{\gamma(0)})$ and $t^{-1}(X_{\gamma(1)})$ to get a Morita equivalence between $\fsl$ and $\fel$ that we call $Q_\gamma$.
\end{definition}

\begin{remark}
    Even though we have used a slightly different presentation of the ideas in the logarithmic case, making explicit use of the criterion for a Morita equivalent subgroupoid \Cref{lemma: criterion for Morita subgroupoid}, the composition of $s^{-1}(X_{\gamma(0)})$ with $t^{-1}(X_{\gamma(1)})$ is exactly given by groupoid arrows in the twisted fundamental groupoid of the total space $\fptl$, which under the map $\fptl \xrightarrow{\pi_*} \Pi_1(\mathbb{D})$ (given by integrating the differential of the map $\pi: (\Gamma^*X, \Gamma^*D) \to \mathbb{D}$) project to the homotopy class of the canonical path $e: t \mapsto (\gamma(t), t)$. This is because the image of such arrows under $\pi_*$ have source $(\gamma(0), 0)$ and target $(\gamma(1), 1)$, and $\mathbb{D}$ is contractible, so any path connecting these two endpoints has to be homotopic to $e$.
\end{remark}

\begin{corollary}
    Let $p: X \to S$ be a holomorphic submersion which is also a fibration, with the \textit{fibres being complex curves}, and let $D$ be a reduced smooth divisor such that the restriction $p|_D: D \to S$ is a cover of $S$. Let $\gamma: I \to S$ be a path in $S$. There is a canonically defined equivalence of categories $\mathbb{P}_\gamma: \text{FlatConn}_{X_{\gamma(0)}}(D_{\gamma(0)}) \to \text{FlatConn}_{X_{\gamma(1)}}(D_{\gamma(1)})$ induced by $Q_\gamma$.
\end{corollary}

\begin{proof}
    By using \Cref{prop: Morita equivalence and representations}, from $Q_\gamma$ we get an equivalence of categories \\ $\text{Rep}(\Pi_1(X_{\gamma(0)}, D_{\gamma(0)})) \xrightarrow{\cong} \text{Rep}(\Pi_1(X_{\gamma(1)}, D_{\gamma(1)}))$. We can then compose this on both sides with the generalized Riemann-Hilbert correspondence of \cite{Stokes groupoids} as written in \Cref{sec: twisted fundamental groupoids} above to get an equivalence of categories $\mathbb{P}_\gamma: \text{FlatConn}_{X_{\gamma(0)}}(D_{\gamma(0)}) \to \text{FlatConn}_{X_{\gamma(1)}}(D_{\gamma(1)})$.
\end{proof}

\subsection{Construction of the morphism of Morita equivalences in the logarithmic case}\label{subsec: log morphism}
We have set up the previous sections to be able to mimic the arguments to construct the morphism of Morita equivalences in the logarithmic case. For this section, fix a homotopy of paths $\gamma \overset{H}{\sim} \gamma'$ fixing endpoints.

\begin{notation}
    Just as we denoted the Morita equivalence in the logarithmic case by $Q_\gamma$ again, we will denote the morphism of Morita equivalences associated to a homotopy $H$ by $\eta_H: Q_\gamma \to Q_{\gamma'}$ again.
\end{notation}

Just as in \Cref{subsec:morphism of Morita equivalences} in the nonsingular case, choose a contractible open subset denoted $\mathbb{D}^2$ of $S \times \mathbb{C}$ which contains $Gr(H)$, with the latter being contractible since it is homeomorphic to the domain of $H$ given by $B^2$. As before, choose this to extend $\mathbb{D}_{\gamma}$ and $\mathbb{D}_{\gamma'}$ which `thicken' $\gamma$ and $\gamma'$ respectively. Then one can form the pullback family:

\[\begin{tikzcd}
	{(H^*X, H^*D) \overset{\text{def}}{=} (X \times_S \mathbb{D}^2, D \times_S \mathbb{D}^2)} & {(X,D)} \\
	{\mathbb{D}^2} & S
	\arrow[from=1-1, to=1-2]
	\arrow["{\pi_H}"', from=1-1, to=2-1]
	\arrow["p", from=1-2, to=2-2]
	\arrow["H"', from=2-1, to=2-2]
\end{tikzcd}\]

The inclusions of pairs
\[(\Gamma^*X, \Gamma^*D) \hookrightarrow (H^*X, H^*D) \hookleftarrow (\Gamma'^*X, \Gamma'^*D)
\]
induce maps of twisted fundamental groupoids
\[ \fptl \xrightarrow{\iota_\gamma} \fhl \xleftarrow{\iota_{\gamma'}} \fptsl\]
These are defined exactly as in \Cref{subsec:log Q_gamma}: one first looks at the maps on logarithmic tangent bundles induced by the differential of the inclusions above, and then integrates these by \Cref{prop: Lie II} to get maps of their source-simply connected integrations. 

We need an analogue of \Cref{lemma: inclusion into homotopy total space is equivalence}, which was exactly the tool needed to construct the morphism of Morita equivalences in the nonsingular case.

\begin{lemma}\label{lemma: inclusion into logarithmic homotopy total space is equivalence}
    The maps $\fptl \xrightarrow{\iota_\gamma} \fhl$ and \\
    $\fhl \xleftarrow{\iota_{\gamma'}} \fptsl$ are fully faithful as functors.
\end{lemma}

\begin{proof}
    Similarly to the proof of \Cref{thm: log Q_gamma is Morita equivalence}, we first apply \Cref{lemma: criterion for Morita subgroupoid} to get a Morita equivalence $t^{-1}(\Gamma^*X)$ between  $\fhl|_{\Gamma^*X}$ and $\fhl$, and then compose this Morita equivalence with the isomorphism $\iota_\gamma: \fptl \xrightarrow{\cong} \fhl|_{\Gamma^*X}$. This is an isomorphism by exactly the argument in the proof of \Cref{thm: log Q_gamma is Morita equivalence}; namely they have isomorphic Lie algebroids, and $\fptl$ is the source-simply connected integration, so by \Cref{prop: Lie II} we can integrate the identity map on Lie algebroids to get a covering map $\fptl \to \fhl|_{\Gamma^*X}$. This covering map is an injection on the dense open given by restricting to the complement of the divisor $\Gamma^*D$, because there it is just a map of fundamental groupoids of the complement for which injectivity follows by \Cref{lemma: inclusion into homotopy total space is equivalence}. A covering map injective on a dense open is an injection everywhere, and it is surjective because it is a covering map, giving us the required isomorphism of Lie groupoids.

\end{proof}

Now that we have a logarithmic analogue of \Cref{lemma: inclusion into homotopy total space is equivalence}, we can reproduce the construction of the morphism $\eta_H$ in the logarithmic setting.

\begin{theorem}\label{thm: log eta_H is morphism}
Let $\fptl \xrightarrow{\iota_\gamma} \fhl \xleftarrow{\iota_{\gamma'}} \fptsl$ be the maps induced by inclusion of pairs as defined above. Let $\eta_H: Q_\gamma \to Q_{\gamma'}$ be defined by $\eta_H(\alpha) = \iota_{\gamma'}^{-1} \circ \iota_\gamma(\alpha)$. Then $\eta_H$ is a morphism of Morita equivalences.    
\end{theorem}

\begin{proof}
    The same argument as for \Cref{eta_H is morphism} works, using \Cref{lemma: inclusion into logarithmic homotopy total space is equivalence} instead of \Cref{lemma: inclusion into homotopy total space is equivalence}.
\end{proof}

\begin{corollary}

A homotopy of paths (fixing endpoints) $\gamma \overset{H}{\sim} \gamma'$ in $S$ induces a natural isomorphism $\hat{\eta}_H: \mathbb{P}_\gamma \Rightarrow \mathbb{P}_{\gamma'}$ between the functors induced by $Q_\gamma$ and $Q_{\gamma'}$.
    
\end{corollary}

\begin{proof}
    This is just \Cref{prop: morphism of Morita equivalences induces natural isomorphism} applied to the morphism of Morita equivalences $\eta_H: Q_\gamma \to Q_{\gamma'}$ produced in \Cref{thm: log eta_H is morphism}.
\end{proof}

These constructions are all functorial in inputs from $\Pi_{\leq 2}(S)$, i.e. given a composition of paths $\gamma.\tilde{\gamma}$, there is an isomorphism of Morita equivalences $Q_{\gamma} \circ Q_{\tilde{\gamma}} \cong Q_{\gamma . \tilde{\gamma}}$, and given a composition of homotopies $H . \tilde{H}$, there is an equality $\eta_H \circ \eta_{\tilde{H}} = \eta_{H.\tilde{H}}$. These isomorphisms are provided by the structure maps of the 2-category of Lie groupoids themselves. For instance $Q_{\gamma \circ \tilde{\gamma}}$ consists of groupoid arrows which project down to the graph of $\gamma \circ \tilde{\gamma}$, which admit a map from $Q_{\gamma} \circ Q_{\tilde{\gamma}}$ since the latter composition was defined as $Q_\gamma \times_{X_{\gamma(1)}} Q_{\tilde{\gamma}} / \fe$. This map is an isomorphism since $Q_\gamma$ consists of groupoid arrows projecting to the graph of $\gamma$ and $Q_{\tilde{\gamma}}$ consists of arrows projecting to the graph of $\tilde{\gamma}$, and if two different compositions of such groupoid arrows map to the same arrow in $Q_{\gamma \circ \tilde{\gamma}}$, they must be related by an arrow in $\fe$, which is exactly what we are quotienting by to produce $Q_\gamma \circ Q_{\tilde{\gamma}}$.

This completes the construction of the 2-functor $\mathbb{P}$ categorifying logarithmic isomonodromy, and finishes the proof of \Cref{thm: main theorem}.

\begin{remark}

From $\gamma$ we produced a canonical equivalence of categories of logarithmic flat connections on the fibres $X_{\gamma(0)}$ and $X_{\gamma(1)}$. This gives that from isomonodromy along $\gamma$ we can produce a canonical isomorphism of the moduli \textit{prestacks} giving logarithmic flat connections on $(X_{\gamma(0)}, D_{\gamma(0)})$ and $(X_{\gamma(1)}, D_{\gamma(1)})$, without even passing to stacks. This upgrades the picture of isomonodromy giving maps between coarse moduli spaces of flat connections with logarithmic singularities.
    
\end{remark}

\section{$\pi_2$-action via the first Hirzebruch surface}\label{sec: Hirzebruch surface}

If we specialize the construction in the main theorem to the constant path $\gamma = e_s$ for $s \in S$, we see that the constant path (which is an identity arrow in $\Pi_{\leq 2}(S)$) must be mapped by $\mathbb{P}$ to the \textit{identity} Morita equivalence of the twisted fundamental groupoid $\Pi_1(X_s, D_s)$ of the fibre over $s$. Note that this is given by the space of arrows of the groupoid itself, with left and right actions corresponding to groupoid multiplication. Evaluating the $\text{Rep}$ functor on this identity Morita equivalence gives us the identity functor on the category of logarithmic flat connections on $X_s$ with singularities along $D_s$.

Looking at the second homotopical level, we get that self-homotopies of the constant path $e_s$, which are elements of $\pi_2(S,s)$, are mapped by $\mathbb{P}$ to automorphisms of the identity Morita equivalence, which are then mapped by $\text{Rep}$ to automorphisms of the identity functor on the category of logarithmic flat connections on $X_s$. Concretely, these are given by families of automorphisms $\hat{\eta}_{(E,\nabla)}: (E, \nabla) \to (E, \nabla)$ which are compatible with any morphism between any two connections. Notice in particular that here we have \textit{not} punctured away the divisor $D$, i.e. the Betti data of such connections is not just given by a monodromy representation, but also has logarithmic singular information along the divisor. We are producing automorphisms of all of this data from the $\pi_2(S,s)$ action. 

If one were to instead puncture away the divisor $D$, we would be getting a map:
\[ \pi_2(S,s) \to \text{Aut}(\Pi_1(X_s\backslash D_s))
\]
Here the right-hand side is automorphisms as a Morita equivalence (and not as a groupoid). If we now pick a point $x$ in the fibre $X_s \backslash D_s$, we can evaluate such an automorphism at the identity arrow at $x$ to get an arrow in $\Pi_1(X_s \backslash D_s)$ whose source and target are both $x$ (since a morphism of Morita equivalences commutes with the moment maps to base of the groupoid, which in this case of the identity Morita equivalence, are just the source and target maps themselves). Such an arrow with both source and target being $x$ is an element of the fundamental \textit{group} $\pi_1(X_s \backslash D_s, x)$. Thus we get a map:
\[\text{Aut}(\Pi_1(X_s\backslash D_s)) \to \pi_1(X_s \backslash D_s, x)\]
that we can compose with the above map from $\pi_2(S,s)$ to get a map:
\[ \pi_2(S,s) \to \pi_1(X_s \backslash D_s, x)\]
This map is just the connecting homomorphism of the fibration $X \backslash D \to S$.

Thus in the logarithmic case, we can interpret the $\pi_2(S,s)$ action as generalizing the connecting homomorphism of such a fibration while keeping track of logarithmic singular data along the divisor $D$.

\begin{example}

For an explicit example, consider $X = \mathbb{F}_1$, $D = D_0 + D_1$, and $S = \mathbb{P}^1$. Here $\mathbb{F}_1$ is the first Hirzebruch surface, defined as the $\mathbb{P}^1$-bundle over $\mathbb{P}^1$ given by $\mathbb{P}(\mathcal{O}_{\mathbb{P}^1} \oplus \mathcal{O}_{\mathbb{P}^1}(1))$.
This is isomorphic to the blow-up of $\mathbb{P}^2$ at a single point. $D_0$ and $D_1$ are the two canonical sections of this bundle given by the projectivizations of the direct summands $\mathcal{O}$ and $\mathcal{O}(1)$ respectively, and their sum gives a reduced smooth divisor $D = D_0 + D_1$ in $\mathbb{F}_1$. This divisor gives two distinct points in every fibre; selecting a point in the base $s \in \mathbb{P}^1$, we can select a coordinate on the fibre $X_s = \mathbb{P}^1$ which identifies the divisor $D_s$ with $0 + \infty$.

If we puncture away the divisor, the fibration is given by:
\[ \mathbb{C}^* \hookrightarrow X^{\circ} = \mathbb{F}_1 \backslash D_0 + D_1 \xrightarrow{p} \mathbb{P}^1 \]

By deformation retracting the fibre $\mathbb{C}^*$ to an $S^1$, we have that the total space $X^{\circ}$ is homotopy equivalent to $S^3$. The connecting homomorphism of this fibration is thus an isomorphism:
\[ \dots \rightarrow \cancelto{0}{\pi_2(X^{\circ}, x)} \to \pi_2(\mathbb{P}^1, s) \xrightarrow{\cong} \pi_1(\mathbb{C}^*, x) \to \cancelto{0}{\pi_1(X^{\circ}, x)} \to \dots
\]

This gives an example of the $\pi_2(S,s)$-action being nontrivial for families of the type we are discussing. If we now do \textit{not} puncture away the divisor $D_0+D_1$, then the discussion about $\pi_2(S,s)$ acting by automorphisms of the identity functor on the category of logarithmic flat connections on the fibre $X_s$ specializes to give a map:
\[ \pi_2(\mathbb{P}^1,s) \to \text{Aut}(\text{Id}_{\text{FlatConn}_{\mathbb{P}^1}(0+\infty)})
\]
where the right-hand side is the identity functor on the category of flat connections on $\mathbb{P}^1$ with logarithmic singularities at $0$ and $\infty$.
    
\end{example}

\section{Future work}\label{sec: future work}

This paper deals with the categorification of isomonodromic deformations of flat connections on complex curves with logarithmic singularities, i.e. for reduced divisor $D$. In a sequel paper, we will aim to categorify isomonodromic deformations for flat connections with \textit{irregular} singularities instead, i.e. for non-reduced divisor $D$.

In \cite{Boalch published thesis}, Philip Boalch shows that isomonodromy gives a flat \textit{symplectic} Ehresmann connection on the bundle of coarse moduli spaces of flat connections with irregular singularities. This means that the map of coarse moduli spaces associated to $\gamma$ is a symplectomorphism. In future work, we hope to extend the symplectic structure to the entire moduli \textit{stack} of flat connections with singularities, in both the logarithmic and irregular case, and show that isomonodromy along a path $\gamma$ gives a symplectic map of such stacks.

\end{document}